\documentclass[11pt,twoside]{amsart}
\usepackage{amsmath,amssymb,amsthm,stmaryrd}
\usepackage{mathrsfs}
\usepackage{nicefrac}
\usepackage{graphicx}
\usepackage[dvipsnames]{xcolor}
\usepackage{hyperref}

\makeatletter
\@namedef{subjclassname@2020}{%
  \textup{2020} Mathematics Subject Classification}
\makeatother

\numberwithin{equation}{section}

\newtheorem{Thm}{Theorem}[section]
\newtheorem{Lem}[Thm]{Lemma}
\newtheorem{slem}[Thm]{Sublemma}
\newtheorem{Pro}[Thm]{Proposition}
\newtheorem{Cor}[Thm]{Corollary}

\newcommand{\R}{{\mathbb R}}
\newcommand{\N}{{\mathbb N}}

\newcommand{\al}{\alpha}

\newcommand{\eps}{\epsilon}

\newcommand{\ga}{\gamma}

\newcommand{\om}{\omega}

\renewcommand{\rho}{\varrho}
\newcommand{\si}{\sigma}

\newcommand{\area}{\operatorname{area}}

\newcommand{\Jac}{\operatorname{Jac}}

\newcommand{\length}{\operatorname{length}}

\renewcommand{\d}{\partial}

\newcommand{\CAT}{\operatorname{CAT}}

\newtheorem*{premainthm}{Theorem}
\newtheorem*{mainthm}{Main Theorem}

\newcommand{\bcor}{\begin{Cor}}
\newcommand{\ecor}{\end{Cor}}
\newcommand{\ben}{\begin{enumerate}}
\newcommand{\bit}{\begin{itemize}}
\newcommand{\blem}{\begin{Lem}}
\newcommand{\bslem}{\begin{slem}}
\newcommand{\bprop}{\begin{Pro}}
\newcommand{\bthm}{\begin{Thm}}
\newcommand{\een}{\end{enumerate}}
\newcommand{\eit}{\end{itemize}}
\newcommand{\elem}{\end{Lem}}
\newcommand{\eslem}{\end{slem}}
\newcommand{\eprop}{\end{Pro}}
\newcommand{\ethm}{\end{Thm}}

\begin{document}

\title[Least-area discs]{Existence of least-area discs}

\author{Alexander Lytchak}
%
\author{Stephan Stadler}
%
\date{\today}

\begin{abstract}
We prove that every closed Lipschitz curve in Euclidean space extends
to a Lipschitz disc of least area, with Lipschitz constant controlled
by that of the boundary curve. Our proof gives a new
approach to the classical Douglas--Radó solution of the Plateau problem but accommodates arbitrary
self-intersections and avoids Sobolev maps.
The argument extends to metric spaces with curvature bounded above
in the sense of Alexandrov.
\end{abstract}

\maketitle

\section{Introduction}

We prove the following Lipschitz version of Plateau's problem.

\begin{premainthm}\label{thm_euclidean}
Every $L$-Lipschitz map $\gamma\colon S^1\to\R^n$ admits a
$5L$-Lipschitz extension $u\colon\bar D\to\R^n$ of least parametrized
area among all Lipschitz extensions of $\gamma$.
\end{premainthm}

The boundary parametrization is prescribed, and the curve may have
arbitrary self-intersections. The proof uses Alexandrov geometry
to obtain the uniform Lipschitz estimate needed for a direct
minimizing-sequence argument.

The classical Plateau problem asks for a disc of least area spanning
a given Jordan curve. Douglas and Radó independently solved this
problem by an ingenious argument based on minimizing Dirichlet energy
\cite{D_plateau,R_plateau}.

Morrey extended this approach to Riemannian targets \cite{M_plateau}. The energy method produces solutions which are continuous on the closed disc and smooth in its interior.
It does not address self-intersecting boundary curves, nor does it produce a Lipschitz solution with a prescribed boundary parametrization.
For self-intersecting curves, Hass constructed singular minimal discs
by a cut-and-paste argument \cite{Ha_sing}. Building on the metric
Plateau theory of Lytchak and Wenger \cite{LWPlateau}, Creutz obtained
Hölder-continuous fillings \cite{Cr_sing}. Our result improves the
regularity to Lipschitz continuity and gives an explicit bound in
terms of the prescribed boundary map.

Our argument naturally allows for targets with curvature
bounded above in the sense of Alexandrov.

\begin{mainthm}\label{thm_main}
Let $X$ be a $\CAT(\kappa)$ space and let $c\colon S^1\to X$ be
$L$-Lipschitz. If $\kappa>0$, assume that the image of $c$ is contained
in a ball of radius at most $\pi/(2\sqrt{\kappa})$.
Then $c$ admits a $10L$-Lipschitz extension
$f\colon\bar D\to X$ of least parametrized area among all Lipschitz
extensions of $c$.
\end{mainthm}

The main obstacle to a direct proof is compactness: a bound on the
areas of Lipschitz fillings gives no control over their Lipschitz
constants. We overcome this by first minimizing area among fillings
with a prescribed Lipschitz bound. We then pass to a length-minimizing
disc below such a minimizer and use its intrinsic geometry to
construct a new parametrization. This preserves the prescribed
boundary map, does not increase area, and has Lipschitz constant
controlled solely by the boundary curve. Letting the initial
Lipschitz bound tend to infinity now yields a least-area filling.

For comparison, White proved that, for $k\geq2$, every Lipschitz map
$S^k\to\R^n$ admits a Lipschitz extension to the ball that minimizes
area among all integral current fillings, and hence among all
Lipschitz ball fillings \cite{Whi-exist}. This stronger conclusion
fails for $k=1$: a filling of higher genus can have strictly smaller
area than every disc filling \cite[p.~56]{Law_min}.

\subsection{AI usage statement}
The proof was found without any use of AI.

\section{Plateau's problem}

\subsection{Notation}

We denote the open unit disc in the Euclidean plane $\R^2$ by $D$ and its closure by $\bar D$.
We denote the distance in a metric space $X$ by $d$ and 
the open ball in $X$ of radius $r$ and center $x\in X$ by $B_r(x)$.
If closed balls in $X$ are compact, then $X$ is called {\em proper}.
For $n\geq 0$, the $n$-dimensional Hausdorff measure on $X$ is denoted by $\mathcal H^n$. 
A continuous curve $c$ in $X$ is called {\em rectifiable} if its length, denoted by $\length(c)$, is finite. A \emph{geodesic} is an isometric map  from an interval into $X$. 
The space $X$ is called \emph{a geodesic space} if any pair of points in $X$ is connected by a geodesic.
Let $\kappa$ be a real number. Recall that a CAT($\kappa$) space is a complete metric space $X$ where any pair of points at distance strictly less than $\pi/\sqrt{\kappa}$ is connected by a geodesic and such that distances between points on a geodesic triangle of perimeter strictly less than $2 \pi/\sqrt{\kappa}$
are bounded above by the distances between corresponding points on the comparison triangle in the simply connected surface of constant curvature $\kappa$. (For $\kappa\leq 0$ the restrictions on distance and perimeter are understood to be vacuous.
We restrict our presentation to the case $\kappa=0$. The general case only requires minor modifications, relying on \cite{LyWa_min} and \cite{LP_retract} instead of \cite{PS}. For background on $\CAT(\kappa)$ geometry we refer the reader to \cite{AKP, ballmannbook, BriH}.


A {\em $\CAT(0)$ disc} $Z$ is a $\CAT(0)$  space homeomorphic to the closed unit disc $\bar D$. Recall that a map $f:Y\to Z$ between topological spaces is called {\em monotone}, if its fibers are connected. 

\blem\label{lem_mon}
Let $Z$ be a $\CAT(0)$  disc 
and $\eta\colon S^1\to\d Z$ an $L$-Lipschitz boundary parametrization. Then there exists a monotone $5L$-Lipschitz map $\mu\colon\bar D\to Z$ extending $\eta$.
\elem

\proof
Let $z$ be a point in the interior of $Z$ and define $\mu\colon\bar D\to Z$ by sending every radial linear segment from $0$ to a point $\theta$ in $S^1$ with constant speed to the geodesic from $z$ to $\eta(\theta)$. Then $\mu$ is monotone by \cite[Lemma~30]{Sta}. To estimate the Lipschitz constant we may rescale $Z$ such that $L=1$. Then the $\CAT(0)$  property implies that $\mu$ does not increase the length of distance circles around $0$. Moreover, $\mu$ can stretch a radial segment at most by a factor equal to the diameter of $Z$ which is at most $\pi$. Thus, for points $v,w\in S^1$ with an angle $\al$ between them and $0\leq t\leq s\leq 1$ we have
\begin{align*}
  d(\mu(tv),\mu(sw))&\leq t\al+(s-t)\pi\\
  &\leq \frac{\pi}{2}|tv-sw|+\pi|tv-sw|=\frac{3\pi}{2}|tv-sw|.
\end{align*}
This proves the claim.
\qed

\subsection{Lipschitz maps and area}

Let $X$ be a complete metric space. A map $u\colon D\to X$ is {\em metrically differentiable} at a point $x\in D$, if there exists a semi-norm $\si$ on $\R^2$ such that 
\[\lim\limits_{|a|+|b|\to 0}\frac{d(u(x+a),u(x+b))-\si(a-b)}{|a|+|b|}=0.\]
We call $\si$ the {\em metric differential} of $u$ at $x$ and denote it by $|du_x|$. 

Let
$u\colon D\to X$ be a Lipschitz map.
By \cite[Theorem~2]{Kirch}, $u$
is metrically differentiable at almost all points in $D$.
Its {\em (parameterized Hausdorff) area} is defined by 
\[\area(u):= \int_D \Jac(|du_x|)\,dx,\] 
where the Jacobian $\Jac(s)$ of a semi-norm $s$ on $\R^2$ is the Hausdorff $2$-measure in $(\R^2, s)$ of the Euclidean unit square if $s$ is a norm and $\Jac(s)=0$ otherwise. By \cite[Theorem~7]{Kirch}, for every Lebesgue measurable subset $A\subset\R^2$, $u$ satisfies the following area formula.
\[\int_A \Jac(|du_x|)\,dx=\int_X \# (u^{-1}(y)\cap A)\, d\mathcal H^2(y).\]
In particular, if $u$ is monotone, then $\#u^{-1}(y)=1$ for $\mathcal H^2$-almost all points $y$ in the image of $u$ and we have 
\[\int_A \Jac(|du_x|)\,dx=\mathcal H^2(u(A)).\]


\blem\label{lem_area_mon}
Let $X$ be a complete metric space and $u^\pm\colon \bar D\to X$  Lipschitz maps. 
Suppose
\[\length(u^-\circ\ga)\leq\length(u^+\circ\ga)\]
holds for every curve $\ga$ in $\bar D$.
Then, $\area(u^-)\leq \area(u^+)$ holds.
\elem

\proof
By \cite[Proposition~4.10]{LWPlateau}, the assumption implies $|du^-_x|\leq |du^+_x|$ for almost all $x\in D$. At each such point, we have 
$\Jac(|du^-_x|)\leq \Jac(|du^+_x|)$ and the claim follows by integration.
\qed
\medskip

The following result is as a replacement for Arzelà-Ascoli in case our space $X$ is not proper.
It is well-known, see \cite[Section~4]{Guo} or \cite{IW_ultra} for the most general result. Its proof requires the concept of ultralimits of metric spaces, we refer the reader to \cite[Section~I.5]{BriH}. For convenience of the reader, we sketch the proof.

\blem\label{lem_AA}
Let $X$ be a $\CAT(0)$ space. Let $c\colon S^1\to X$ be a $L$-Lipschitz map and $u_k\colon \bar D\to X$ a sequence of $L$-Lipschitz extensions of $c$ with $\lim_{k\to\infty}\area(u_k)=A$. Then, there exists an $L$-Lipschitz extension $u\colon \bar D\to X$ of $c$ with $\area(u)\leq A$.
\elem

\proof[Sketch of proof]
An ultralimit $u_\om=\om\lim u_k$ is a $L$-Lipschitz
 into the ultracompletion of $X$, $u_\om\colon\bar D\to X_\om$. The map $u_\om$ can be realized as follows. We can choose isometric embeddings $\iota_k$ of the images of $u_k$ into a common complete metric space $Y$ and produce a uniform limit $v\colon\bar D\to Y$ of a subsequence $(\iota_{k_l}\circ u_{k_l})$.  The image of $v$ isometrically embeds into $X_\om$ via a map $j$. Then, $u_\om=j\circ v$. By lower semicontinuity of area \cite[Corollary~5.8]{LWPlateau}, we have $\area(u_\om)=\area(v)\leq\liminf_{l\to\infty}\area(u_{k_l})=A$. Since $X$ embeds canonically as a closed convex subset $X\hookrightarrow X_\om$, there is a $1$-Lipschitz retraction $\pi\colon X_\om\to X$.
 Thus, $u=\pi\circ u_\om$ is as required.
\qed

\blem\label{lem_M_lip}
Let $X$ be a $\CAT(0)$ space. Let $c\colon S^1\to X$ be an $L$-Lipschitz map. Then for every $M\geq L$ there exists an $M$-Lipschitz extension $u_M\colon \bar D\to X$ of least area among all $M$-Lipschitz extensions. 
\elem

\proof
By Kirszbraun's theorem \cite{LS_kirszbraun}, there exists an $L$-Lipschitz extension $u\colon \bar D\to X$.
Let $A\geq 0$ be the infimum of areas of $M$-Lipschitz extensions of $c$. Let $u_k\colon\bar D\to X$ be a sequence of $M$-Lipschitz extensions of $c$ with $\lim\limits_{k\to\infty}\area(u_k)=A$. Then Lemma~\ref{lem_AA} completes the proof.
\qed

\subsection{Length-minimizing discs}

We refer the reader to \cite{BBI, LWint,PS} for a more general  discussion of the following construction.
Let $Y$ be a geodesic space, $X$ a  metric space and $u\colon Y\to X$ a Lipschitz map. The {\em intrinsic distance associated with $u$}
is the function $d_u:Y\times Y\to [0,\infty ]$
 defined by
\[d_u(y_1,y_2)=\inf\{\length(u\circ\ga)\}\]
where the infimum is taken over all curves $\ga$ in $Y$ connecting $y_1$ and $y_2$.
It defines a pseudo-metric and the associated metric space $Z_u$
which arises from  identifying pairs of points at zero $d_u$-distance, is a length
space.  We will call it the \emph{intrinsic metric space (associated with the map $u$)}.

By construction, the space $Z_u$ comes with a canonical, surjective Lipschitz projection $\pi_u\colon Y\to Z_u$ and
a $1$-Lipschitz map $\bar u:Z_u\to X$ such that $u=\bar u\circ \pi_u$.

Let $X$ be a metric space.
A Lipschitz map $v\colon \bar D\to X$ is called {\em length-minimizing}, if for every continuous map $w\colon \bar D\to X$ which satisfies $w|_{S^1}=v|_{S^1}$ the condition
\[\length(w\circ\ga)\leq\length(v\circ\ga)\]
for every curve $\ga$ in $\bar D$ can only hold if equality holds for every curve $\ga$.
The following result produces a length-minimizing disc below a given Lipschitz disc, see \cite[Lemma~5.1]{LyWa_min} and \cite[Proposition~5.1]{PS}.

\blem\label{lem_exist}
Let $X$ be a $\CAT(0)$ space and $u\colon\bar D\to X$ a Lipschitz map. Then there exists a length-minimizing map  $v\colon\bar D\to X$ with
\[\length(v\circ\ga)\leq\length(u\circ\ga)\]
for every curve $\ga$ in $\bar D$.
\elem


The next result is a special case of \cite[Main Theorem and Proposition~3.1]{PS}. In the present setting, the proof given in \cite{PS} becomes considerably shorter and simpler, since our maps are assumed a priori to be Lipschitz.


\bthm\label{thm_memi}
Let $X$ be a $\CAT(0)$ space and $u\colon\bar D\to X$ a Lipschitz continuous length-minimizing map with embedded boundary $u|_{S^1}$. Then the natural factorization $u= \bar u\circ\pi_u$ through the intrinsic space $Z_u$ satisfies:
\begin{itemize}
\item $Z_u$ is a $\CAT(0)$ disc;
    \item $\pi_u$ is a surjective monotone Lipschitz map with
\[\length(u\circ\ga)=\length(\pi_u\circ\ga)\]
    for every curve $\ga$ in $\bar D$, in particular, $\area(u)=\area(\pi_u)=\mathcal H^2(Z_u)$.
\end{itemize}
\ethm

\subsection{Main Theorem}

In this section, we prove the Main Theorem in the case of $\CAT(0)$ spaces. The modifications required for $\CAT(\kappa)$ spaces with $\kappa\neq 0$ are standard, except that in Lemma~\ref{lem_AA} one uses the Lytchak--Petrunin retraction \cite[Theorem~1.1]{LP_retract} in place of the nearest-point projection.

\proof[Proof of Main Theorem for $\CAT(0)$]
Let $c\colon S^1\to X$ be an $L$-Lipschitz map.
We first treat the case where $c$ is an embedding.
Fix a natural number $N\geq 5L$.
By Lemma~\ref{lem_M_lip}, there exists a Lipschitz map $u\colon\bar D\to X$ with $u|_{S^1}=c$ which minimizes area among all $N$-Lipschitz extensions of $c$. By Lemma~\ref{lem_exist}, there exists a length-minimizing map  $v\colon\bar D\to X$ with $v|_{S^1}=u|_{S^1}=c$ and such that 
\[\length(v\circ\ga)\leq\length(u\circ\ga)\]
for every curve $\ga$ in $\bar D$. We infer from Lemma~\ref{lem_area_mon} that $\area(v)\leq\area(u)$. Thus $v$ is also area minimizing among $N$-Lipschitz extensions of $c$.  By Theorem~\ref{thm_memi},
$v=\bar v\circ\pi_v$ with 
\[\pi_v\colon\bar D\to Z_v\ \text{ and }\ \bar v\colon Z_v\to X.\]
The intrinsic space $Z_v$ is a $\CAT(0)$ disc with $\mathcal H^1(\d Z_v)=\length(c)$ and $\mathcal H^2(Z_v)=\area(v)$. 
$\pi_v$ restricts to an $L$-Lipschitz parametrization of $\d Z_v$.
Lemma~\ref{lem_mon} provides a surjective monotone $5L$-Lipschitz map $\mu\colon \bar D\to Z_v$ extending $\pi_v|_{S^1}$. Thus, the map $f_N\colon \bar D\to X$ with $f_N=\bar v\circ \mu$ is a $5L$-Lipschitz extension of $c$ and satisfies 
\[\area(f_N)\leq\area(\mu)=\area(v).\]
In particular, since $5L\leq N$, $f_N$ minimizes area among all $N$-Lipschitz extensions of $c$.
We define $f\colon\bar D\to X$ as a subsequential  uniform limit of the sequence $(f_N)$. Note that $\area(f_N)\geq\area(f_{N+1})$ for all $N\in\N$. Lower semicontinuity of area (\cite[Corollary~5.8]{LWPlateau}) implies
\[\area(f)\leq\inf_{n\in\N}\area(f_N).\]
It follows that $f$ is a $5L$-Lipschitz continuous extension of $c$ which minimizes area among all Lipschitz extensions and the proof for embedded curves $c$ is complete.

Now let $c$ be arbitrary. Choose $\eps_k\to 0$ and define
$c_k(\theta)=(c(\theta),\eps_k\theta)$
as auxiliary curves in the $\CAT(0)$ space $X\times\mathbb R^2$. The curves $c_k$
are embedded and $(L+\eps_k)$-Lipschitz. Let
\[
g_k:\bar D\to X\times\mathbb R^2
\]
be $M$-Lipschitz extensions of $c_k$ minimizing area among all Lipschitz extensions of $c_k$.
Since $c_k$ has image in the closed convex set $X\times \bar B_{\eps_k}(0)$, and since the nearest point projection to
a closed convex subset in a $\CAT(0)$ space is $1$-Lipschitz, we may assume that $g_k$ has image in $X\times \bar B_{\eps_k}(0)$. Passing to a subsequence, we obtain a uniform limit map $g\colon \bar D\to X$ which is $5L$-Lipschitz and extends $c$.
The curves $c_k$ and $c$ bound Lipschitz annuli whose areas converge to zero as $k\to\infty$.
Since $g_k$ are area minimizing extensions so is $g$. The proof is complete.
\qed

\bibliographystyle{alpha}
\bibliography{plateau}


\end{document}